\documentclass[11pt,reqno]{amsart}
\usepackage{hyperref}
\usepackage{amssymb,amsmath,amsfonts,latexsym}
\usepackage{bm}
\usepackage{mathtools}
\allowdisplaybreaks
\usepackage{enumerate}

\usepackage{pst-node}
\usepackage{tikz-cd}

\usepackage{array,graphics,color}
\numberwithin{equation}{section}
\newtheorem{dfn}{Definition}[section]
\newtheorem{thm}[dfn]{Theorem}
\newtheorem{lma}[dfn]{Lemma}

\newtheorem{crlre}[dfn]{Corollary}
\newtheorem{xmpl}[dfn]{Example}
\newtheorem{rmrk}[dfn]{Remark}
\newcommand{\N}{\mathbb{N}}
\newcommand{\D}{\mathbb{D}}
\newcommand{\C}{\mathbb{C}}		
\newcommand{\fcl}{\mathcal{F}}
\newcommand{\dcl}{\mathcal{D}}
\newcommand{\ncl}{\mathcal{N}}
\newcommand{\mcl}{\mathcal{M}}
\newcommand{\wcl}{\mathcal{W}}
\newcommand{\hilh}{\mathcal{H}}

\newcommand{\kcl}{\mathcal{K}}

\newcommand{\hdcr}{{H}^2_{\mathbb{C}^r}(\mathbb{D})}
\newcommand{\hdcp}{{H}^2_{\mathbb{C}^p}(\mathbb{D})}
\newcommand{\hdcrp}{H^2_{\mathbb{C}^{r+p}}(\mathbb{D})}
\newcommand{\hdcm}{{H}^2_{\mathbb{C}^m}(\mathbb{D})}
\newcommand{\hdcc}{{H}^2_{\mathbb{C}}(\mathbb{D})}

\DeclarePairedDelimiterX{\norm}[1]{\lVert}{\rVert}{#1}

\begin{document}

\title[STUDY OF NEARLY INVARIANT SUBSPACES WITH FINITE DEFECT IN HILBERT SPACES]{STUDY OF NEARLY INVARIANT SUBSPACES WITH FINITE DEFECT IN HILBERT SPACES }


\author[Chattopadhyay] {Arup Chattopadhyay}
\address{Department of Mathematics, Indian Institute of Technology Guwahati, Guwahati, 781039, India}
\email{arupchatt@iitg.ac.in, 2003arupchattopadhyay@gmail.com}

\author[Das]{Soma Das}
\address{Department of Mathematics, Indian Institute of Technology Guwahati, Guwahati, 781039, India}
\email{soma18@iitg.ac.in, dsoma994@gmail.com}




\begin{abstract}
In this article, we briefly describe nearly $T^{-1}$ invariant subspaces with finite defect for a shift operator $T$ having finite multiplicity acting on a separable Hilbert space $\mathcal{H}$ as a generalization of nearly $T^{-1}$ invariant subspaces introduced by Liang and Partington in \cite{YP}. In other words  we  characterize nearly $T^{-1}$ invariant subspaces with finite defect in terms of backward shift invariant subspaces in vector-valued Hardy spaces by using Theorem 3.5 in \cite{CDP}.
Furthermore,  we also provide a concrete representation of the nearly $T_B^{-1}$ invariant subspaces with finite defect in a scale of Dirichlet-type spaces $\mathcal{D}_\alpha$ for $\alpha \in [-1,1]$ corresponding to any finite Blashcke product $B$ .

\end{abstract}
\maketitle

\section{Introduction}
The structure of the invariant subspaces of an operator $T$ plays an important role to study the action of $T$ on the full space in a better way.
 To that aim,  the study of (almost) invariant subspaces were initiated and a suitable investigation of these brings  the concept such as near invariance. The study of nearly invariant subspaces for the backward shift in the scalar valued Hardy space $\hdcc$ were introduced by Hayashi \cite{HE}, Hitt\cite{HIT}, and then Sarason \cite{SR} in the context of kernels of Toeplitz operators. Going further, Chalendar-Chevrot-Partington (C-C-P) \cite{CCP} gives a complete characterization of nearly invariant subspaces under the backward shift operator acting on the vector-valued Hardy space, providing
a vectorial generalization of a result of Hitt. In 2004,  Erard investigated the  nearly invariant subspaces related to  multiplication operators in Hilbert spaces of analytic functions in \cite{ER}. The concept of nearly invariant subspaces of finite defect for the backward shift in the scalar valued Hardy space was introduced by Chalendar- Gallardo-Partington (C-G-P) in \cite{CGP} and provides a complete characterization of these spaces in terms of backward shift invariant subspaces. A recent
preprint \cite{CDP} by the authors of this article along with C. Pradhan characterizes nearly invariant subspace of finite defect for the backward shift operator acting on the vector-valued Hardy space and provides a vectorial generalization of C-G-P
algorithm. In this connection we
also mention that similar type of connection also obtained independently by R. O\textquoteright Loughlin in \cite{OR}. Recently, Liang and Partington introduce the notion of nearly $T^{-1}$ invariant subspaces in general Hilbert space setting \cite{YP} and provide a  representation of  nearly $T^{-1}$ invariant subspaces for the shift operator $T$ with finite multiplicity acting on a separable infinite dimensional Hilbert space $\mathcal{H}$ in terms of backward shift invariant subspaces on the vector valued Hardy spaces as an application of Corollary 4.5. given in \cite{CCP}. Moreover, they 
also give a description of the nearly $T_B^{-1}$ invariant subspaces for the operator $T_B$ of multiplication by $B$ in a scale of Dirichlet-type spaces \cite{YP}, where $B$ is any finite Blashcke product.

Motivated by the work of Liang and Partington in \cite{YP}, we also introduce the notion of nearly $T^{-1}$ invariant subspaces with finite defect (see Definition 2.1) for an left invertible operator $T$ acting on a separable infinite dimensional Hilbert space as a generalization of nearly $T^{-1}$ invariant subspaces. The purpose of this article is to study nearly $T^{-1}$ invariant subspaces with finite defect for a shift operator $T$ with finite multiplicity acting on a separable Hilbert space. In other words we provide a characterization of nearly $T^{-1}$ invariant subspaces with finite defect in terms of backward shift invariant subspaces in vector-valued Hardy spaces by using our recent Theorem 3.5 (C-D-P) in \cite{CDP}. Moreover, we also give a concrete representation of the nearly $T_B^{-1}$ invariant subspaces with finite defect in a scale of Dirichlet-type spaces $\mathcal{D}_\alpha$ for $\alpha \in [-1,1]$ corresponding to any finite Blashcke product $B$ by extending some results of C. Erard in \cite{ER}. There are also many other contributions related with this topic and the interested reader can also refer to \cite{BS}\cite{FM} and the references therein. In order to state the precise contribution of this paper, we need to recapitulate some useful notations and definitions. 

Let $\hilh$ be a separable infinite dimensional Hilbert space and $\mathcal{B}(\hilh)$ denote the set of all bounded linear operators acting on $\hilh$. The $\mathbb{C}^m$- valued Hardy space \cite{JP} over the unit disc $\mathbb{D}$ is denoted by $\hdcm$ and defined by $$\hdcm:=\Big\{F(z)=\sum_{n\geq 0} A_nz^n:~\|F\|^2= \sum_{n\geq 0}~\norm{A_n}_{\mathbb{C}^m}^2<\infty,~A_n\in\mathbb{C}^m\Big\}.$$ We can also view the above Hilbert space as the direct sum of $m$-copies of $H^2_{\mathbb{C}}(\D)$ or  sometimes it is useful to see the above space as a tensor product of two Hilbert spaces $H^2_{\mathbb{C}}(\D)$ and $\mathbb{C}^m$, that is, $$H^2_{\mathbb{C}^m}(\D)  \equiv   \underbrace{H^2_{\mathbb{C}}(\D)\oplus \cdots\oplus H^2_{\mathbb{C}}(\D)}_{m}\equiv H^2_{\mathbb{C}}(\D)\otimes \mathbb{C}^m.$$
On the other hand the space $\hdcm$ can also be defined as the collection of all $\mathbb{C}^m$-valued analytic functions $F$ on $\mathbb{D}$ such that 
$$\norm{F}= \Big[~\sup_{0\leq r<1} \frac{1}{2\pi} \int_0^{2\pi} \vert F(re^{i\theta})\vert^2~ d\theta~\Big]^{\frac{1}{2}}<\infty.$$ 
Moreover the nontangential boundary limit (or radial limit) $$F(e^{i\theta}):= \lim\limits_{r\rightarrow 1-}F(re^{i\theta})$$ exists almost everywhere on the unit circle $\mathbb{T}$ (for more details see \cite{NB}, I.3.11). Therefore $\hdcm$ can be embedded isomertically as a closed subspace of
$L^2(\mathbb{T},\mathbb{C}^m)$ by identifying $\hdcm$ through the nontangential boundary limits of the $\hdcm$ functions.
Let $S$ denote the forward shift operator  (multiplication by the independent variable)
acting on $\hdcm$, that is, $SF(z)=zF(z)$,~$z\in \D$. The adjoint of $S$ is denoted by $S^*$ and defined in $\hdcm$ as the operator 
$$S^*(F)(z)=\dfrac{F(z)-F(0)}{z},~~ F\in\hdcm$$ which is known as backward shift operator.
The Banach space of all $\mathcal{L}(\mathbb{C}^r,\mathbb{C}^m)$ (set of all bounded linear operators from $\mathbb{C}^r$ to $\mathbb{C}^m$)- valued  bounded analytic functions on $\mathbb{D}$ is denoted by 
$H^{\infty}_{\mathcal{L}(\mathbb{C}^r,\mathbb{C}^m)}(\D)$ and the associated norm is $$\norm F_{\infty} = \sup_{z\in \D} ~\norm {F(z)}.$$ Moreover, the space 
$H^{\infty}_{\mathcal{L}(\mathbb{C}^r,\mathbb{C}^m))}(\D)$ can be embedded isometrically as a closed subspace of $L^\infty(\mathbb{T},\mathcal{L}(\mathbb{C}^r,\mathbb{C}^m))$.
Note that each $\Theta\in H^{\infty}_{\mathcal{L}(\mathbb{C}^r,\mathbb{C}^m)}(\D)$ induces a bounded linear map $T_{\Theta}\in H^{\infty}_{\mathcal{L}(\mathbb{C}^r,\mathbb{C}^m)}(\D)$ defined by
$$T_{\Theta}F(z)=\Theta(z)F(z).~~(F\in H^{2}_{\mathbb{C}^r}(\D))$$
The elements of $H^{\infty}_{\mathcal{L}(\mathbb{C}^r,\mathbb{C}^m)}(\D)$  are called the \emph{multipliers} and are determined by
$$\Theta\in H^{\infty}_{\mathcal{L}(\mathbb{C}^r,\mathbb{C}^m)}(\D)\textit{~if~and~only~if~} ST_{\Theta}=T_{\Theta}S,$$
where the shift $S$ on the left hand side and the right hand side act on $ H ^{2}_{\mathbb{C}^m}(\D)$ and $ H ^{2}_{\mathbb{C}^r}(\D)$ respectively. A multiplier $\Theta \in  H ^{\infty}_{\mathcal{L}(\mathbb{C}^r,\mathbb{C}^m)}(\D)$ is said to be \emph{inner} if $T_{\Theta}$ is an isometry, or equivalently, 
$\Theta(e^{it})\in \mathcal{L}(\mathbb{C}^r,\mathbb{C}^m)$ is an isometry almost everywhere with respect to the Lebesgue measure on $\mathbb{T}$. Inner multipliers are among the most important tools for classifying invariant subspaces of reproducing kernel Hilbert spaces. For instance: 

\begin{thm}\label{a1}
	(Beurling-Lax-Halmos \cite{NF}) 
	A non-zero closed subspace $\mathcal{M}\subseteq  H ^{2}_{\mathbb{C}^m}(\D)$ is shift
	invariant if and only if there exists an inner multiplier $\Theta \in H ^{\infty}_{\mathcal{L}(\mathbb{C}^r,\mathbb{C}^m)}(\D)$ such that
	$$\mathcal{M} = \Theta  H ^{2}_{\mathbb{C}^r}(\D),$$
	for some $r$ ($1\leq r\leq m$).  
\end{thm}
Consequently, the space $\mathcal{M}^{\perp}$ of 
$H^2_{\mathbb{C}^m}(\D)$ which is invariant under $S^*$ (backward shift) can be represented as 
$$\mathcal{K}_{\Theta}: = \mathcal{M}^{\perp} =  H^2_{\mathbb{C}^m}(\D) \ominus \Theta  H^2_{\mathbb{C}^r}(\D),$$
which also known as model spaces (\cite{EMV1,EMV2,NB,RSN}).
Let $P_m:L^2(\mathbb{T},\mathbb{C}^m) \to H^2_{C^m}(\D)$ be an orthogonal projection onto $H^2_{C^m}(\D)$ defined by $$ \sum_{n=-\infty}^{\infty}A_ne^{int}\mapsto\sum_{n=0}^{\infty}A_ne^{int}.$$ Therefore $P_m(F)=(Pf_1,Pf_2,\ldots ,Pf_m),$ where $P$ is the Riesz projection on $\hdcc$ \cite{EMV1} and $F=(f_1,f_2,\ldots,f_m)\in L^2(\mathbb{T},\mathbb{C}^m)$.
Also note that for any $\Phi \in L^\infty(\mathbb{T},\mathcal{L}(\mathbb{C}^m,\mathbb{C}^m))$, the Toeplitz operator $T_{\Phi}:\hdcm \to \hdcm$ is defined by $$T_{\Phi}(F)=P_m(\Phi F)$$ for any $F\in \hdcm$.
Next we introduce a special family of Hilbert spaces of analytic functions. Let $\alpha$ be any real number. Then the Dirichlet-type spaces are denoted by $\dcl _\alpha\equiv \dcl _\alpha(\D)$ and defined by
$$\dcl _\alpha \equiv \dcl _\alpha(\D) :=\Big\{f:\D \to \C :f(z)=\sum_{n=0}^{\infty}a_nz^n,\sum_{n=0}^{\infty}(n+1)^\alpha |a_n|^2<\infty \Big\}.$$
Then each $\dcl_\alpha$ is a Hilbert space with respect to the  norm 
$$\norm f_\alpha :=\bigg(\sum_{n=0}^{\infty}(n+1)^\alpha |a_n|^2 \bigg)^{\frac{1}{2}}.$$
Note that the particular instances of $\alpha$ yield well-known Hilbert spaces of analytic functions on $\D$. More precisely, when $\alpha =0$ we get the Hardy space $\hdcc$, for $\alpha =-1$ we have the classical Bergman space $\mathcal{A}^2$,  and $\alpha =1$ correspond to the Dirichlet space $\dcl$. Since $\norm f _\gamma <\norm f_\beta$ for $\gamma < \beta$, then the continuous inclusion  $\dcl _\beta \subset \dcl _\gamma$ holds for any $\gamma <\beta $. For more information about Dirichlet-type spaces we refer to \cite{BS} and the references therein. Recall that an analytic function $u$ is said to be an multiplier of $\dcl _\alpha$ if for any $f\in \dcl _\alpha$, $uf \in \dcl_\alpha$ that is, the analytic Toeplitz operator $T_u: f\to uf $ is defined everywhere on $\dcl_\alpha$ (hence bounded by closed graph theorem). Furthermore, one can easily check that any finite Blaschke product $B$ is a multiplier for each $\dcl_\alpha$ spaces. Note that a finite Blaschke product is given by
\[
 B(z)= e^{i\theta}\prod_{k=1}^N
 \dfrac{z-z_k}{1-\overline{z_k}z}, \quad (z\in \D)
\]
where $\alpha_i\in\D$ and the degree of $B$ is just the number of zeros $\{z_1,\ldots,z_N\}$, counted with multiplicity. Moreover, finite Blashcke products play an important role in mathematics. We refer \cite{TG} and \cite{GM} for more on the subject of multipliers of $\mathcal{D}_{\alpha}$ and the qualitative study of finite Blaschke product respectively.  
The famous Wold Decomposition Theorem \cite{CJ} implies that for any Blaschke product $B$, each element $f\in \hdcc$ has the following decomposition:
$$f(z)=\sum_{n=0}^{\infty}B^n(z)h_n(z),$$
where $h_n$ belongs to the model space $\mathcal{K}_B=\hdcc \ominus B\hdcc$. An analogous theorem for Dirichlet-type spaces $\dcl_\alpha(\D)$ is the following:
\begin{thm}\cite[Theorem 3.1]{CGP2}\cite[Theorem 2.1]{CGP1}\label{th1}
	
	Suppose $\alpha \in [-1,1]$ and $B $ is a finite Blaschke product. Then $f\in \dcl _\alpha(\D)$ if and only if $f=\sum_{n=0}^{\infty}B^nh_n$ (convergence in $\dcl_\alpha (\D)$ norm) with $h_n\in \mathcal{K}_B= \hdcc \ominus B\hdcc$ and 
	\begin{equation}
	\sum_{n=0}^{\infty }(n+1)^\alpha\norm {h_n}^2 _{H^2}<\infty.
	\end{equation}
	Moreover, since $B$ is a finite Blaschke product, then $\mathcal{K}_B$ is finite dimensional and hence we can consider other (equivalent) norms here, such as $\norm h _{\dcl_\alpha}$.
\end{thm}
The nearly invariant subspaces related to the multiplication operator $M_u$ in the Hilbert space of analytic functions has been studied by C. Erard in \cite{ER}. In fact Erard gave the definition of \textquotedblleft nearly invariant under division by u \textquotedblright,
which is same as \textquotedblleft nearly $M_u^{-1}$ invariant\textquotedblright, a special case of the notion of nearly $T^{-1}$ invariant subspaces for any left invertible operator $T\in \mathcal{B}(\mathcal{H})$ recently introduced by Liang and Partington in \cite{YP} and the definition is the following:
\begin{dfn}\cite[Definition 1.2]{YP}
Let $\hilh$ be a separable infinite dimensional Hilbert space and $T\in \mathcal{B}(\hilh ) $ be left invertible. Then a closed subspace $\mcl \subset \hilh $ is said to be nearly $T^{-1}$ invariant if for every $g\in \hilh $ such that $Tg\in \mcl $ then it holds that $g\in \mcl $.	
\end{dfn}

It is well known that the shift operator acting on a separable Hilbert space is a generalization of the unilateral shift $S$ and the operator $T_B$ on $\hdcm$. Recall that, an operator $T \in \mathcal{B}(\hilh)$ is said to be a shift opertor if it is an isometry and $T^*$ converges strongly to zero that is, $\norm {T^{*n}h}\to 0$ as $n\to \infty$ for all $h\in \hilh$ \cite{RR}. Equivalently, an isometry $T\in \mathcal{B}(\mathcal{H})$ is a shift operator if and only if  $T$ is pure that is, $\cap_{n=0}^\infty T^n\hilh =\{0\}$. Therefore it is easy to observe that shift operator is an isometry and left invertible. Moreover, the multiplicity of a shift operator $T\in \mathcal{B}(\mathcal{H})$  is defined to be the dimension of $Ker T^* = \hilh \ominus T\hilh $.  
As we have discussed earlier, Liang and Partington have characterized nearly $T^{-1}$ invariant subspaces for a shift operator $T\in \mathcal{B}(\mathcal{H})$ with finite multiplicity and furthermore  they also studied the nearly $T_B^{-1}$ invariant subspaces corresponding to a  finite Blaschke product $B$ in a scale of Dirichlet-type spaces $\dcl_\alpha$ for $\alpha \in [-1,1]$ in \cite{YP}. The main aim of this article is to first introduce the notion of nearly $T^{-1}$ invariant subspaces with finite defect for a shift operator $T\in \mathcal{B}(\mathcal{H})$ with finite multiplicity and then characterize those subspaces  in terms of backward shift invariant subspaces in vector-valued Hardy spaces. Furthermore, we also study the nearly $T_B^{-1}$ invariant subspaces in a scale of Dirichlet-type spaces $\dcl_\alpha$ for $\alpha \in [-1,1]$  corresponding to a  finite Blaschke product $B$ and provide a concrete representation of it by generalizing some results of C. Erard \cite{ER} in our context.

The rest of the paper is organized as follows: In Section 2, we introduce the notion of nearly $T^{-1}$ invariant subspaces with finite defect for an left invertible operator $T\in \mathcal{B}(\mathcal{H})$ and give a chracterization of nearly $T^{-1}$ invariant subspaces with finite defect for the shift operator $T\in \mathcal{B}(\mathcal{H})$ with finite multiplicity. In Section 3, we deal with the study of nearly $T_B^{-1}$ invariant subspaces with finite defect corresponding to a finite Blaschke product $B$ in a scale of Dirichlet-type spaces $\dcl_\alpha$ for $\alpha \in [-1,1]$.

\section{Characterization of Nearly Invariant Subspaces with finite defect for the Shift Operator}

In this section, we study nearly $T^{-1}$ invariant subspaces with finite defect for a shift operator $T\in \mathcal{B}(\hilh)$ having finite multiplicity. Now we introduce the notion of nearly $T^{-1}$ invariant subspaces with finite defect for any left invertible operator $T\in \mathcal{B}(\hilh)$ as a generalization of nearly $T^{-1}$ invariant subspaces.
\begin{dfn}
	Let $T\in \mathcal{B}(\hilh)$ be left invertible. Then a closed subspace $\mcl$ of $\hilh$ is said to be nearly $T^{-1}$ invariant with finite defect $p$ if there exists a $p$ dimensional subspace $\fcl $ (which may be taken to be  orthogonal to $\mcl$) such that for any $f\in \hilh $ with $Tf\in \mcl $, then it holds that $f\in \mcl \oplus \fcl$. 
\end{dfn}
The following lemma gives a connection of nearly invariant subspaces with same defect between similar operators.
\begin{lma}\label{a1}
	Let  $T_1 \in \mathcal{B}(\hilh _1)$ and $T_2\in \mathcal{B}(\hilh _2)$ be two left invertible operators such that they are similar by some invertible operator $V:\hilh _1 \rightarrow \hilh _2$, so that $T_2=VT_1V^{-1}$. Let $\mcl $ be a nearly $T_1^{-1}$ invariant subspace with defect $p$ in $\hilh _1$; then $V(\mcl)$ is also a nearly $T_2^{-1}$ invariant subspace with the same defect $p$ in $\hilh _2$.
\end{lma}
\begin{proof}
	Suppose $g\in \hilh_2$ such that  $T_2 g \in V\mcl $, then we want to show $g\in V\mcl \oplus V\fcl$, where $\fcl$ is the $p$ dimensional defect space for $\mcl$ in $\hilh_1$.
	Since $T_2g =VT_1V^{-1}g \in V\mcl$, then it implies that $T_1V^{-1}g \in V\mcl$. Moreover, since $\mcl $ is nearly $T_1^{-1}$ invariant with defect space $\fcl$, then we must have $V^{-1}g \in \mcl \oplus \fcl $. Thus $g\in V(\mcl \oplus \fcl)= V\mcl \oplus V\fcl$, proving that $V(\mcl)$ is a nearly $T_2^{-1}$ invariant subspace with defect $p$ in $\hilh _2$.
\end{proof}
 Now onwards we always assume $T\in \mathcal{B}(\hilh)$ is a shift operator with multiplicity $m$ throughout this section. Let $\{e_1,e_2,\ldots ,e_m  \}$ be an orthonormal basis of $\kcl =\hilh \ominus T\hilh$ and let  $\delta _j^m =(0,0,\ldots ,1,\ldots ,0)$ with $1$ in the jth place be an orthonormal basis of $\kcl _z =\hdcm \ominus z\hdcm$ for $j=1,2,\ldots ,m$. By considering the following two orthogonal decompositions 
$$ \hilh =\bigoplus _{i=0}^\infty T^i\kcl ~\text{and}~ \hdcm =\bigoplus _{i=0}^\infty z^i\kcl_z   ,$$
 we have an unitary mapping $U :\hilh \to \hdcm $ defined by 
\begin{equation}\label{equni}
U(T^ie_j)=z^i\delta ^{m}_j.
\end{equation}
Therefore the following  diagram \ref{diag1} corresponding to the shift operator $T:\hilh \to \hilh $ with multiplicity $m$ and the unilateral shift $S:\hdcm \to \hdcm$ is commutative.  

\begin{align}\label{diag1}
\begin{matrix}
\hilh\xrightarrow[\hspace*{3cm}]{T}\hilh\\
U\Bigg\downarrow\hspace*{3.5cm}\Bigg\downarrow U\\
\hdcm\xrightarrow[\hspace*{3cm}]{S}\hdcm
\end{matrix}
\end{align}
Therefore from the above commutative diagram \ref{diag1} we get
\begin{equation}\label{d}
S^nU= UT^n , \forall n\in \N\cup \{0\}.
\end{equation}
 Now onwards we denote by $P_\mcl$ as the orthogonal projection of $\hilh$ onto a closed subspace $\mcl$ of $\hilh$. The following lemma gives an upper bound concerning the dimension of the subspace 
 $\mcl \ominus (\mcl \cap T\hilh )$:
\begin{lma}\label{a}
	Let $T\in \mathcal{B}(\mathcal{H})$ be a shift operator with multiplicity m and let $\mcl$ be a non trivial closed subspace of $\hilh$ such that $\mcl \nsubseteq T\hilh $ (that means $\mcl$ is not properly contained in $T\hilh$). Then
	\begin{equation}
	1\leq r:=dim(\mcl \ominus (\mcl \cap T\hilh ))\leq m.
	\end{equation}
\end{lma}
\begin{proof}
Since $T$ is a shift operator with multiplicity $m$, then $dim(\hilh \ominus T\hilh)=m$. Moreover, since $\mcl \nsubseteq T\hilh $, then $\mcl \ominus (\mcl \cap T\hilh ) \neq \{0\}$. Let $\{e_1,\ldots ,e_m  \}$ be an orthonormal basis of $\hilh \ominus T\hilh$. Our claim is that $\{P_{\mcl} e_1,\ldots ,P_{\mcl} e_m  \}$ generates $\mcl \ominus (\mcl \cap T\hilh )$. Indeed, for any $g\in \mcl \ominus (\mcl \cap T\hilh )$	with $\langle g,P_\mcl e_i\rangle =0$ for all $ i\in\{1,\ldots ,m\}$ implies $g=0$ and hence  $1\leq r:=dim(\mcl \ominus (\mcl \cap T\hilh ))\leq m$. 
\end{proof}
Next by using condition \ref{d} and the above lemma \ref{a} we have the following result.
\begin{lma}\label{b}
	Let $\mcl$ be a non trivial nearly $T^{-1}$ invariant subspace with finite defect $p$ and let $G_0=[g_1,g_2,\ldots ,g_r]^t$ be an $r\times 1$ matrix with $\{g_1,g_2,\ldots ,g_r\}$ is an orthonormal basis of $\mcl \ominus (\mcl \cap T\hilh )$ (note that the superscript $t$ denotes the transpose of a matrix). Then $F_0=[Ug_1,Ug_2,\ldots ,Ug_r]^t$ be an $r\times m$ matrix with $\{Ug_1,Ug_2,\ldots ,Ug_r\}$ is an orthonormal basis for $U\mcl \ominus (U\mcl \cap z\hdcm )$.
\end{lma}
\begin{proof}
The proof is straightforward and we leave it to the reader.	
\end{proof}

Going further, we need the following useful lemma due to Liang and Partington in \cite{YP}.
\begin{lma}\cite[Lemma 2.3]{YP}\label{c}
Suppose that $T : \hilh \to \hilh$ is a shift operator, and let $U$ be as \ref{d}. Then
	\begin{equation}
	U^*[(Ug)h]=h(T)g,
	\end{equation}
	for any $g\in \hilh ,h\in \hdcc $.
\end{lma}
 Now we are in a position to state and prove our main result in this section which provides an isometric relation between nearly $T^{-1}$ invariant subspaces with defect $p$ and the backward shift invariant subspaces of $\hdcrp = \hdcr \times \hdcp$.
\begin{thm}\label{mainth1}
	Suppose $T$ is a shift operator with multiplicity $m$
and $\mcl \subset  \hilh$ is a non trivial nearly $T^{-1}$ invariant subspace with defect $p$ and let $\fcl$ be the corresponding $p$ dimensional defect space. Let $F_1=[f_1,f_2,\ldots ,f_p]^t$ be a $p\times 1$ matrix containing an orthonormal basis $\{f_1,f_2,\ldots ,f_p\}$ of $\fcl$.Then\vspace*{0.1in}\\
(i) in the case when $\mcl \nsubseteq T\hilh$, there exista a non negative integer $r^\prime \leq r+p$ and an inner multiplier $\Phi \in H^\infty_{\mathcal{L}(\C^{r^\prime },\C^{r+p}))} (\D)$, unique upto an unitary equivalence such that
\begin{equation}
\mcl= \Big\{f\in \hilh : f= K_0(T)G_0+ TK_1(T)F_1 : (K_0,K_1)\in\hdcrp \ominus \Phi  H^2_{\C ^{r^\prime }}(\D) \Big\},
\end{equation}
where $G_0=[g_1,g_2,\ldots ,g_r]^t$ is an $r\times 1$ matrix with $\{g_1,g_2,\ldots ,g_r\}$ is an orthonormal basis of $\mcl \ominus (\mcl \cap T\hilh )$ and also there exists an isometry $$Q:\mcl \to \hdcrp \quad \text{defined by}\quad Q(f)=(K_0,K_1).$$
\noindent (ii) In the case, when $\mcl \subseteq T\hilh$, there exists a non negative integer $p^\prime \leq p$ and an inner multiplier $\Theta \in H^\infty_{\mathcal{L}(\C^{p^\prime},\C^p))} (\D)$ which is unique upto unitary constant such that
\begin{equation}
\mcl= \Big\{f\in \hilh : f=  TK_1(T)F_1 : K_1\in\hdcp \ominus \  H^2_{\C ^{p^\prime }}(\D) \Big\},
\end{equation}
and also there exists an isometry $$R:\mcl \to \hdcp \quad \text{defined by} \quad R(f)=K_1. $$
\end{thm}
\begin{proof}
From Lemma \ref{a1} and using \ref{equni}, we say $U\mcl$ is a nearly $S^*$ invariant subspace of $\hdcm $ with defect $p$ and the correosponding defect space is $U\fcl \subseteq \hdcm $.
Therefore by applying our recent Theorem 3.5 (C-D-P) in \cite[Theorem 3.5, case (i)]{CDP} corresponding to nearly $S^*$ invariant subspace with finite defect in vector valued Hardy space $\hdcm $, we have
$$U\mcl =\Bigg\{F\in\hdcm :F(z)= F_0(z)^tK_0(z)+ \sum_{j=1}^{p} zk_j(z)Uf_j(z) : (K_0,k_1,\ldots,k_p)\in \kcl   \Bigg\}, $$ 
where 
$\kcl \subset \hdcr \times \underbrace{\hdcc\times\cdots\times \hdcc}_{p} $ is a closed $S^*\oplus\cdots\oplus S^*$- invariant subspace of the vector valued Hardy space $ H ^2(\mathbb{D},\mathbb{C}^{r+p})$, 
\begin{equation}\label{e}
\norm{F}^2=\norm{K_0}^2+\sum_{j=1}^{p}\norm{k_j}^2,
\end{equation}
 and $F_0$ given in Lemma \ref{b}. Therefore by Beurling-Lax-Halmos theorem on $\hdcrp$, there exists a non negative integer $r^\prime \leq r+p$ and an inner multiplier 
 $\Phi \in H^\infty_{\mathcal{L}(\C^{r^\prime},\C^{r+p}))}(\D) $ unique upto unitary equivalance such that $\kcl =\hdcrp \ominus \Phi H^2_{\C^{r^\prime}}(\D)$. Thus if we consider $f\in \mcl$, then there exists $(K_0,k_1,k_2,\ldots , k_p )\in \kcl$ such that 
 $$Uf=[Ug_1,Ug_2,\ldots ,Ug_r]K_0 +\sum_{j=1}^{p} Sk_jUf_j$$
and
\begin{equation}\label{f}
\norm f ^2=\norm {Uf}^2=\norm{K_0}^2+\sum_{j=1}^{p}\norm{k_j}^2 .
\end{equation}
 Let $K_0=(k^0_1,k^0_2,\ldots ,k^0_r)\in \hdcr,$ then
\begin{align*}
Uf &=[Ug_1,Ug_2,\ldots ,Ug_r]K_0 +\sum_{j=1}^{p} Sk_jUf_j =\sum_{i=1}^{r}(Ug_i)k^0_i +\sum_{j=1}^{p} SUf_jk_j 
\end{align*}
and therefore by using Lemma \ref{c} we get
\begin{align*}
Uf &= \sum_{i=1}^{r}(Ug_i)k^0_i +\sum_{j=1}^{p} SUf_jk_j 
= \sum_{i=1}^{r}U(k^0_i(T)g_i) +\sum_{j=1}^{p} U(Tf_j)k_j \\
  &= \sum_{i=1}^{r}U(k^0_i(T)g_i) +\sum_{j=1}^{p}U(k_j(T)Tf_j) = U(\sum_{i=1}^{r}k^0_i(T)g_i +\sum_{j=1}^{p}T(k_j(T)f_j))\\
  &= U(K_0(T)G_0 + TK_1(T)F_1),
\end{align*}
and hence
$$f=K_0(T)G_0 + TK_1(T)F_1, $$
where $K_1=(k_1,k_2,\ldots ,k_p) \in \hdcp$. 
Therefore $$\mcl= \Big\{f\in \hilh : f= K_0(T)G_0+ TK_1(T)F_1 : (K_0,K_1)\in\hdcrp \ominus \Phi  H^2_{\C ^{r^\prime }}(\D) \Big\}.$$
Moreover, the relation \ref{e} gives the existence of an isometry $V: U\mcl \to \hdcrp \ominus \Phi  H^2_{\C ^{r^\prime }}(\D)  $. Now if we define $Q=VU$, then $Q:\mcl \to \hdcrp \ominus \Phi  H^2_{\C ^{r^\prime }}(\D) $ is an isometry and the isometric relation is given by \ref{f} . This completes the proof of $(i)$.

For case (ii), we assume $\mcl \subset T\hilh$ and hence $\mcl \ominus (\mcl \cap T\hilh )=\{0\}$. Therefore again by applying C-D-P Theorem \cite[Theorem 3.5, case (ii)]{CDP} we have 
$$U\mcl =\Bigg\{F\in\hdcm :F(z)=  \sum_{j=1}^{p} zk_j(z)Uf_j(z) : (k_1,\ldots,k_p)\in \kcl   \Bigg\} $$ 
where 
$\kcl \subset  \underbrace{\hdcc\times\cdots\times \hdcc}_{p} $ is a closed $S^*\oplus\cdots\oplus S^*$- invariant subspace of the vector valued Hardy space $ H ^2(\mathbb{D},\mathbb{C}^{p})$ and 
\begin{equation}\label{g}
\norm{F}^2=\sum_{j=1}^{p}\norm{k_j}^2 .
\end{equation}
Similarly as in case $(i)$, there exists a non negative integer $p^\prime \leq p$ and an inner multiplier $\Theta \in H^\infty_{\mathcal{L}(\C^{p^\prime},\C^{p})}(\D)$ unique upto an unitary equivalance such that $\kcl =\hdcp \ominus \Phi H^2_{\C^{p^\prime}}(\D)$. Moreover, if $K_1=(k_1,k_2,\ldots ,k_p) \in \hdcp$, then $$\mcl= \Big\{f\in \hilh : f= TK_1(T)F_1 : K_1\in\hdcp \ominus \Theta  H^2_{\C ^{p^\prime }}(\D) \Big\}.$$ 
Furthermore, the equation \ref{g} gives an existence of an isometry $W: U\mcl \to \hdcp \ominus \Phi  H^2_{\C ^{p^\prime }}(\D)  $ and therefore, if we define $R=WU$, then $R:\mcl \to \hdcp \ominus \Theta  H^2_{\C ^{p^\prime }}(\D) $ is an isometry. This completes the proof of $(ii)$.
\end{proof}

The following corollary characterize the nearly $T_B^{-1}$ invariant subspace with finite defect $p$ in $\hdcc$ as a consequence of the above Theorem \ref{mainth1}. Note that for any finite Blaschke $B$ with degree $m$, the operator $T_B :\hdcc \to \hdcc $ is a shift operator with multiplicity $m$.
\begin{crlre}\label{h}
	Let $\mcl \subset \hdcc $ be a non trivial nearly $T_B^{-1}$ invariant subspace with defect $p$, where $B$ is a finite Blaschke of degree $m$ having atleast one zero in $\D\setminus \{0\}$. Let $G_0=[g_1,g_2,\ldots ,g_r]^t$ be an $r\times 1$ matrix with $\{g_1,g_2,\ldots ,g_r\}$ is an orthonormal basis of $\mcl \ominus (\mcl \cap T_B\hilh )$ and let $F_1=[f_1,f_2,\ldots ,f_p]^t$ be a $p\times 1$ matrix containing an orthonormal basis $\{f_1,f_2,\ldots ,f_p\}$ of the defct space $\fcl$. Then there exists a non negative integer $r^\prime \leq r+p$ and an inner multiplier $\Phi \in H^\infty_{\mathcal{L}(\C^{r^\prime },\C^{r+p})} (\D)$, unique upto unitary equivalence such that
	\begin{equation}\label{cr}
	\mcl= \Big\{f\in \hilh : f= K_0(T_B)G_0+ T_BK_1(T_B)F_1 : (K_0,K_1)\in\hdcrp \ominus \Phi  H^2_{\C ^{r^\prime }}(\D) \Big\}.
	\end{equation}
\end{crlre}
The following example gives a better understanding of the above corollary. 
\begin{xmpl}
Let us define $B _a(z)=\dfrac{a-z}{1-\overline{a}z}$ for any $a\in \D\setminus\{0\}$. Now consider the subspace $$\mcl = B_a(z).\bigg\{\bigvee\{1,z^2,z^6,z^8,z^{10},\ldots \} \oplus \bigvee\{z,z^3,z^5,\cdots ,z^{2m+1} \}  \bigg\}$$	for some $m\in \N \cup \{0\}$. Then $\mcl$ is a nearly $T^*_{z^2}$ invariant subspace of $\hdcc$ with defect $1$. It is easy to observe that $dim(\mcl \ominus (\mcl \cap T_{z^2}\hdcc ))=2$, $G_0=B_a(z).[1,z]^t$ and the defect space is $\fcl =\langle z^4\phi _a(z)\rangle $ with$F_1=[z^4\phi _a(z)]$. Therefore for any $f\in \mcl$, we have 
$$f(z) =\bigg[\sum_{k=0}^{\infty}a_{k1}z^{2k},\sum_{k=0}^{\infty}a_{k2}z^{2k}\bigg]G_0(z)+ T_{z^2}\bigg[\sum_{k=0}^{\infty}b_{k}z^{2k}\bigg]F_1,$$ where the constants $a_{k1},a_{k2}$ and $b_k$ satisfies the following:
\begin{equation*}
\begin{cases*}
a_{k1}\in \C ~for~ k\in \{0,1\} ~and~ a_{k1}=0 ~for~ k\geq 2, \\
a_{k2}\in \C ~for~ k\in \{0,1,\cdots ,m\} ~and~ a_{k2}=0 ~for~ k\geq m+1, \\
b_k \in \C ~for~ k\geq 0 .
\end{cases*}
\end{equation*}
Moreover, the equation \ref{cr} along with above discussions conclude
\begin{equation*}
\mcl= \Big\{f\in \hilh : f= K_0(T_{z^2})G_0+ T_{z^2}K_1(T_{z^2})F_1 : (K_0,K_1)\in H^2_{\C ^{2+1}}(\D) \ominus \Phi  \hdcc \Big\},
\end{equation*}
where $\Phi \in  H^\infty_{\mathcal{L}(\C,\C^3)} (\D)$ is an inner multiplier such that $\Phi (z)=(z^2,z^{m+1},0)\in \C ^3$.

\end{xmpl}
\section{Description of Nearly $T_B^{-1}$ Invariant Subspces with Defect for Finite Blaschke $B$ in $\dcl_\alpha$ Spaces}
In this section we discuss about nearly $T_B^{-1}$ invariant subspaces with finite defect corresponding to any finite Blaschke product $B$ in a scale of $\dcl_\alpha$ spaces for $\alpha \in [-1,1]$. Recall that any finite Blaschke product $B$ is a multiplier of each $\dcl_\alpha$, that is the multiplication operator $T_B :\dcl_\alpha \to \dcl_\alpha$ is defined everywhere and  bounded. Moreover, the operator  $T_B$ is bounded below but not an isometry. We refer to the reader concerning the work of Lance and Stessin \cite{LS} in connection with the study of multiplication invariant subspaces of Hardy spaces. In \cite{ER} C. Erard studied the nearly invariant subspaces corresponding to lower bounded multiplication operator $M_u$ on the Hilbert space of analytic functions $\hilh$ and there are four conditions concerning the pairs $(\hilh,u)$ which are as follows: 
\begin{enumerate}[(i)]
	\item $\hilh$ is a Hilbert space and a linear subspace of $\mathcal{O}(\wcl): = \big\{f:\wcl \to \C| ~f~ \text{is analytic} \big\} $,
	where $\wcl$ is an open subset of $\C ^d ~(d\in \N)$,
	\item $u\in \mathcal{O}(\wcl)$  satisfies $uh\in \hilh  $ for all $h\in \hilh$,
	\item for all $w\in \wcl$ the evaluation $\hilh \to \C$, $h\to h(w)$ is continuous, 
	\item there exists $c > 0$ such that for all $h \in \hilh$ $c\norm h _\hilh\leq \norm{uh}_\hilh$.
\end{enumerate}
Corresponding to the above pair $(\hilh ,u)$,  the \emph{lower bound} of the multiplication operator $M_u$ relative to the norm $\norm{.} _\hilh$ is defined by
\begin{equation}\label{lbd}
\gamma _{\hilh ,M_u}= sup\{c>0:\forall h\in \hilh , c\norm{h}_\hilh\leq \norm{uh}_\hilh \} \in (0,\infty).
\end{equation} 
For simplicity we denote $\gamma _{\hilh ,M_u}$ by $\gamma$. In particular for the pair $(\hilh, u(z)=z)$, Erard gives a connection between nearly backward shift invariant subspaces in $\hilh$ and a backward shift invariant subspaces in $H^2_{\C}(\D)$ (see Theorem 5.1 in \cite{ER}). Note that the operator $T_B:\dcl_{\alpha}\rightarrow \dcl_{\alpha}$ is more general than $M_z:H^2_{\C}(\D)\rightarrow H^2_{\C}(\D)$ and the characterizations for nearly $T_B^{-1}$ invariant subspaces in $\dcl_{\alpha}$ for $\alpha\in [-1,1]$ corresponding to the finite Blaschke product $B$ is due to Liang and Partington (see Theorem 3.4 and Theorem 3.7 in \cite{YP} ) by applying some results of Erard \cite{ER}. Here our main aim is to characterize nearly $T_B^{-1}$ invariant subspaces with finite defect in $\dcl_{\alpha}$ for $\alpha\in [-1,1]$ corresponding to the finite Blaschke product $B$. To achieve our goal we need to first extend two important results (namely Approximation Lemma and Factorization Theorem) due to Erard \cite{ER}. Before we proceed note that  if $T:\hilh \to \hilh $ is a bounded operator that is bounded from below, then $T$ has closed range and $T^*T$ is invertible. The following lemma is a generalization of Lemma 2.1. in \cite{ER}.

\begin{lma}[Approximation Lemma]\label{lm}
Let $\hilh$ be a Hilbert space and let $T:\hilh \to \hilh $	be a bounded operator such that for all $h\in \hilh $, $\norm h _\hilh \leq \norm {Th} _\hilh $. Suppose $\mcl$ is a nearly $T^{-1}$ invariant subspace of $\hilh$ with defect $p$ (i.e. the dimension of the defect space $\fcl$ is  $p$). We set $R=(T^*T)^{-1}T^*P_{\mcl \cap T\hilh}$, $Q=P_{\mcl \ominus (\mcl \cap T\hilh)}$, $S=P_\fcl$. Then $\norm R \leq 1$, and for all $h\in \mcl $ and $m\in \N$, we have 
\begin{equation}\label{eqlmma1}
h=\sum_{k=0}^{m}T^kQR^kh +T^{m+1}R^{m+1}+T\sum_{k=1}^{m}T^{k-1}SR^kh 
\end{equation}
and
\begin{equation}\label{eqlmma2}
 \norm h ^2_\hilh \geq \sum_{k=0}^{\infty}\norm {QR^kh} _\hilh ^2 +\sum_{k=1}^{\infty} \norm {SR^kh} ^2_\hilh .
\end{equation}
\end{lma}
\begin{proof}
Consider $h\in \hilh $ and write $P_{\mcl \cap T\hilh}(h) =Th_0$. Then we have 
\begin{equation}\label{i}
TRh =T(T^*T)^{-1}T^*Th_0=Th_0=P_{\mcl \cap T\hilh }(h).
\end{equation}  
Thus for any $h\in \hilh $, we have $\norm {Rh}\leq \norm {TRh} =\norm {P_{\mcl \cap T\hilh}(h)}\leq \norm {h}$ and hence $\norm R \leq 1$.
Suppose $h\in \mcl$ and therefore by using \ref{i} we conclude that  $TRh \in \mcl$. Since $\mcl$ is a nearly $T^{-1}$ invariant subspace with defect $p$, then we have 
\begin{equation}\label{j}
Rh \in \mcl \oplus \fcl.
\end{equation} 
Moreover, by using \ref{i} and since $T$ is bounded below we have for any $h\in \mcl$,
\begin{align}
h &=Qh + TRh \label{k}
\end{align}
and 
\begin{align}
\norm h^2 &\geq \norm {Qh}^2 +\norm {TRh}^2 \geq \norm{Qh}^2 +\norm{Rh}^2. \label{l}
\end{align}
Since $Rh \in \mcl \oplus \fcl$ (by \ref{j}), then we have $$Rh= P_\mcl Rh +SRh$$
which implies that $Rh- SRh \in \mcl$. 
Note that since \ref{k} is true for any $h \in \mcl$, therefore if we replace $h$ by $Rh-SRh$ in \ref{k} we get 
\begin{equation}\label{m}
Rh =QRh +TR^2h +SRh.
\end{equation}
Now it is easy to observe that $R(\mcl \oplus \fcl) \subset \mcl \oplus \fcl$ and hence $R^mh \in \mcl \oplus \fcl ,~\forall m \in \N$. Therefore by induction  from $\ref{m} $ we get for any $m\in \N$,
\begin{equation}\label{esmn1}
R^mh= QR^mh +TR^{m+1}h +SR^mh
\end{equation}
and since $T$ is bounded below we have
\begin{equation}\label{esm1}
\norm {R^m h}^2 \geq \norm {QR^m h}^2 + \norm{R^{m+1}h}^2+ \norm {SR^m h} ^2.
\end{equation}
Finally by combining \ref{k} and \ref{esmn1} we have 
\begin{equation*}
h=\sum_{k=0}^{m}T^kQR^kh + T^{m+1}R^{m+1}h +T\sum_{k=1}^{m}T^{k-1}SR^kh ,~~m\in \N
\end{equation*}
and moreover equations \ref{l} and \ref{esm1} yield that
\begin{equation*}
\norm h ^2_\hilh \geq \sum_{k=0}^{\infty}\norm {QR^kh} _\hilh ^2 +\sum_{k=1}^{\infty} \norm {SR^kh} ^2_\hilh.
\end{equation*}
This completes the proof.
\end{proof}

\begin{rmrk}\label{rem}
	Under the same assumtion as in Lemma \ref{lm}, let $\mcl$ be a nearly $T^{-1}$ invariant subspace of $\hilh $ with defect $p$ such that $\mcl \subseteq T\hilh $ and let $\fcl $ be the corresponding $p$ dimensional defect space having an orthonormal basis $\{e_j\}_{j=1}^p$. Then for any $h\in \mcl$ and $m\in \N$ we have 
	\begin{equation}
	h=T^{m+1}R^{m+1}+T\sum_{k=1}^{m}T^{k-1}SR^kh 
	\quad \text{and}\quad 
	\norm {h}^2_\hilh \geq \sum_{k=1}^{\infty}\norm{SR^kh}^2_\hilh.
	\end{equation} 
\end{rmrk}

Next we denote $D(0,a):=\{z\in\C :|z|<a \}$. As an application of the above Approximation Lemma we have the following theorem which is a generalization of Theorem 3.2 in \cite{ER}.

\begin{thm}[Factorization Theorem]\label{thm1}
Assume that the pair $(\hilh ,u)$ satisfies the four conditions (i)-(iv) given above. Let  $\mcl$ be a nearly $M_u^{-1}$ invariant subspace of $\hilh$ with defect $p$ and let $\fcl$ be the corresponding defect space. Let $\{g_i \}_{i\in I}$ be an orthonormal basis of $\mcl \ominus (\mcl \cap M_u\hilh)$ and let $\{e_j\}_{j=1}^p$ be an orthonormal basis of $\fcl$. Moreover, we also assume that 
	\begin{equation} \label{hypo1}
	 \bigcap_{n\in \N}u^n|_{u^{-1}(D(0,\gamma))}\hilh |_{u^{-1}(D(0,\gamma))} =\{0\},
	 \end{equation}
	where $\hilh |_{u^{-1}(D(0,\gamma))}$ consists of the restrictions to $u^{-1}(D(0,\gamma))$ of the functions of $\hilh$.
	Then\\  (i) in the case when $\mcl \nsubseteq M_u\hilh$, for all $h\in \mcl$, there exist $(q_i)_{i\in I}$ and $(h_j)_{j=1}^p$ in $\mathcal{O}(u^{-1}(D(0,\gamma)))$ such that $$h= \sum_{i\in I}g_iq_i +\gamma ^{-1}M_u\sum_{j=1}^{p}e_jh_j$$ on $u^{-1}(D(0,\gamma))$ for all $i\in I$ and $j\in \{1,\ldots p\}$, and also there exist $(c_{ki})_{k\in \N _0} \in \C^\N$ and $(b_{kj})_{k\in \N }\in \C^\N$, where $\N _0=\N\cup \{0\}$ with 
	\begin{align}
	&q_i =\sum_{k=0}^{\infty}c_{ki}\bigg(\dfrac{u}{\gamma}\bigg)^k ,~ h_j=\sum_{k=1}^{\infty}b_{kj}\bigg( \dfrac{u}{\gamma}\bigg)^{k-1}
	\end{align}
	and
	\begin{align}
	&\sum_{i\in I}\sum_{k=0}^{\infty}|c_{ki}|^2 +\sum_{j=1}^{p}\sum_{k=1}^{\infty}|b_{kj}|^2 \leq \norm h_{\hilh}^2.
	\end{align}
	
	(ii) In the case when $\mcl \subseteq M_u\hilh$, then for all $h\in \mcl$ there exists $(h_j)_{j=1}^p$ in $\mathcal{O}(u^{-1}(D(0,\gamma)))$ such that $$h= \gamma ^{-1}M_u\sum_{j=1}^{p}e_jh_j \quad \text{on}~ u^{-1}(D(0,\gamma)) $$ for all $j\in \{1,2,\ldots p\}$ and also there exists $(b_{kj})_{k\in \N }\in \C^\N$ such that
	\begin{align*}
    h_j=\sum_{k=1}^{\infty}b_{kj}\bigg( \dfrac{u}{\gamma}\bigg)^{k-1} \quad \text{and} \quad 
	\sum_{j=1}^{p}\sum_{k=1}^{\infty}|b_{kj}|^2 \leq \norm h^2 .
	\end{align*}
\end{thm} 
\begin{proof}
(i) First we consider $T=\gamma ^{1}M_u$. Then $T$ satisfies the hypothesis of Lemma \ref{lm}. Now we define $R,Q,S$ as in Lemma \ref{lm} and let $h\in \mcl$.  Then we  define a family of sequences $\{(c_{ki})_{k\in \N_0}\}_{i\in I},$ $\{(b_{kj})_{k\in \N}\}_{j=1}^p$ of complex numbers by the following equations 
\begin{align*}
& QR^kh =\sum_{i\in I}c_{ki}g_i,  \quad k\in \N_0 \quad \text{and} \quad 
SR^kh=\sum_{j=1}^{p}b_{kj}e_j \quad k \in \N.
\end{align*} 
Therefore by using \eqref{eqlmma1} and \eqref{eqlmma2} we get
\begin{align*}
h &= \sum_{k=0}^{m}T^kQR^kh +T^{m+1}R^{m+1}h +\sum_{k=1}^{m}T^kSR^kh \\
&= \sum_{k=0}^{m}\sum_{i\in I}c_{ki}T^kg_i +T^{m+1}R^{m+1}h +T\sum_{k=1}^{m}\sum_{j=1}^{p}b_{kj}T^{k-1}e_j,
\end{align*}
and hence
\begin{align}
h &= \sum_{k=0}^{m}\sum_{i\in I}c_{ki}\bigg(\dfrac{u}{\gamma}\bigg)^kg_i +\bigg(\dfrac{u}{\gamma}\bigg)^{m+1}R^{m+1}h +\gamma ^{-1}u\sum_{k=1}^{m}\sum_{j=1}^{p}b_{kj}\bigg(\dfrac{u}{\gamma}\bigg)^{k-1}e_j, \label{o}
\end{align}
and
\begin{align}
 &\sum_{i\in I}\sum_{k=0}^{\infty}|c_{ki}|^2 +\sum_{j=1}^{p}\sum_{k=1}^{\infty}|b_{kj}|^2 \leq \norm h^2, \label{n}
\end{align}
so that for all $i\in I$ and $j\in \{1,2,\ldots ,p \}$, 
\begin{equation*}
\sum_{k=0}^{\infty}|c_{ki}|^2  <\infty \quad \text{and} \quad  \sum_{k=1}^{\infty}|b_{kj}|^2 < \infty .
\end{equation*}
Therefore it follows that for all $i\in I$, the series $\sum\limits_{k=0}^{\infty}c_{ki}\bigg(\dfrac{u}{\gamma}\bigg)^k$ converges uniformly on compact subsets of $u^{-1}(D(0,\gamma))$, so that its sum, which we  denote by $q_i$, belongs to  $\mathcal{O}(u^{-1}(D(0,\gamma)))$. Similarly the series $\sum\limits_{k=1}^{\infty}b_{kj}\bigg( \dfrac{u}{\gamma}\bigg)^{k-1}$ also converges uniformly on compact subsets of $u^{-1}(D(0,\gamma))$ and hence the sum of the series denoted by $h_j$ also belongs to $\mathcal{O}(u^{-1}(D(0,\gamma)))$. Let $w\in u^{-1}(D(0,\gamma))$, then by using Cauchy-Schwarz inequality and \ref{n} we obtain
\begin{align*}
\sum _{i\in I}|(g_iq_i)(w)| &\leq \bigg(\sum_{i\in I}|g_i(w)|^2\bigg)^{\frac{1}{2}}\bigg(\sum_{i\in I}|q_i(w)|^2\bigg)^{\frac{1}{2}}\\
& \leq\bigg(\sum_{i\in I}|\langle g_i,k_w\rangle|^2\bigg)^{\frac{1}{2}}\bigg(\sum_{i\in I}(\sum_{k=0}^{\infty}|c_{ki}|^2)(\sum_{k=0}^{\infty}\dfrac{|u(w)|^{2k}}{\gamma ^{2k}})\bigg)^{\frac{1}{2}} \leq\norm{Qk_w} _\hilh \norm{h}_\hilh\dfrac{1}{\sqrt{1-\dfrac{|u(w)|^2}{\gamma ^2}}},
\end{align*}
and
\begin{align*}
\sum _{j=1}^{p}|(e_jh_j)(w)| &\leq \norm{Sk_w} _\hilh \norm{h}_\hilh\dfrac{1}{\sqrt{1-\dfrac{|u(w)|^2}{\gamma ^2}}},
\end{align*}
and hence that both the series $\sum\limits _{i\in I}g_iq_i$ and $\sum\limits _{j=1}^{p}e_jh_j$ converges at each point of $u^{-1}(D(0,\gamma))$. Now from equation \ref{o} we obtain $$(h-\sum_{i\in I}g_iq_i-\sum_{j=1}^{p}e_jh_j)|_{u^{-1}(D(0,\gamma))} \in \bigcap_{m\in \N}u^m|_{u^{-1}(D(0,\gamma))}\hilh |_{u^{-1}(D(0,\gamma))},$$
which along with the hypothesis \eqref{hypo1} implies that 
\begin{align*}
h &= \sum_{i\in I}g_iq_i +\gamma ^{-1}M_u\sum_{j=1}^{p}e_jh_j ~\text{on}~ u^{-1}(D(0,\gamma))\quad \text{and} \quad 
\sum_{i\in I}\sum_{k=0}^{\infty}|c_{ki}|^2 +\sum_{j=1}^{p}\sum_{k=1}^{\infty}|b_{kj}|^2 \leq \norm h^2.
\end{align*}
(ii) Again we consider $T=\gamma_1^{-1}M_u$. Therefore by using Remark $\ref{rem}$ and proceeding as in case $(i)$ we obtain
\begin{align*}
h &= \gamma ^{-1}M_u\sum_{j=1}^{p}e_jh_j ~on~ u^{-1}(D(0,\gamma)) \quad \text{and} \quad 
 &\sum_{j=1}^{p}\sum_{k=1}^{\infty}|b_{kj}|^2 \leq \norm h^2.
\end{align*}
 \end{proof}

Now we are in a position to describe the nearly $T_B^{-1}$ invariant subspaces with defect $p$ corresponding to a finite Blaschke $B$ in Dirichlet type spaces $\dcl_\alpha$ for $\alpha \in [-1,1]$ by applying similar type of mechanism done by Liang and Partington in \cite{YP} . Now on wards we assume that $B$ is Blaschke product of degree m and therefore for any non trivial nearly $T_B^{-1}$ invariant subspace $\mcl$ in $\dcl_\alpha$ with defect $p$ and $\mcl \nsubseteq T_B\dcl_\alpha$ we have $$	1\leq r:=dim(\mcl \ominus (\mcl \cap T_B\dcl_\alpha ))\leq m $$
which follows by similar argument as in Lemma \ref{a}. In the sequel, we now endow the space $\dcl_\alpha$ with two different equivalent norms according to the cases $\alpha \in [-1,0)$ and $\alpha \in [0,1]$ and hence we divide the analysis into two subsections.

\subsection{\boldmath $\alpha \in [-1,0)$} 
Note that we need to endow the space $\dcl_\alpha$ with a norm in such a way so that we can a get a nice lower bound of the operator $T_B$. Keeping this information in our mind we endow the space $\dcl_\alpha$ for  $\alpha \in [-1,0)$ with the modified 
equivalent norm denoted by $\|\cdot \|_1$ as follows: for any $f=\sum_{n=0}^{\infty}f_n B^n$ with $f_n \in \mathcal{K}_B$,
\begin{equation}
\norm f_1^2 :=\sum_{n=0}^{G-1}G^\alpha \norm {f_n}^2_{\hdcc}  +\sum_{n=G}^{\infty}(n+1)^\alpha\norm {f_n}^2_{\hdcc},  
\end{equation}
where $G$ is a fixed and sufficiently large positive number to be specified below. It is easy to observe that the lower bound of $T_B$ defined in \ref{lbd} is 
\begin{equation}\label{lbd1}
\gamma _1:=\Bigg(1-\dfrac{1}{G+1} \Bigg)^{-\alpha /2}.
\end{equation}
Thus from the definition of lower bound it follows that for any $f\in \dcl_\alpha$,
\begin{equation*}
\norm {T_Bf}_1^2 =\norm {Bf}_1^2 \geq \gamma _1^2\norm f _1^2
\end{equation*}
and hence the operator $T:=\gamma_1^{-1}T_B:\dcl_\alpha \to \dcl_\alpha $ satisfies
\begin{equation*}
\norm {Tf}_1^2 =\norm{\gamma_1^{-1}T_Bf}_1^2 \geq \norm f _1^2 \quad \text{for any} \quad f\in \dcl_\alpha.
\end{equation*}
Note that the pair $(\dcl_\alpha ,T_B)$ also satisfies conditions (i)-(iv) with lower bound $\gamma_1$ given in \ref{lbd1}. Now we choose $G$ large enough so that $\gamma_1$ satisfies $B^{-1}(D(0,\gamma_1))\supset s\D$ with $s\D$ a disc containing all the zeros of $B$ which ensures that 
\begin{equation}\label{inq}
\norm {\gamma_1 ^{-1}B}_{H^\infty(s\D)} <1.
\end{equation}
Moreover, the operator $T:=\gamma_1 ^{-1}T_B$ satisfies all the assumptions in Lemma \ref{lm} together with the fact that  
$$\bigcap_{m\in \N}B^m\dcl_\alpha |_{s\D}=\bigcap_{m\in \N}T^m\dcl_\alpha |_{s\D}=\{0\}.$$
Combining the above facts together with Theorem \ref{thm1} implies the following lemma, providing a generalization of Lemma 3.6 in \cite{YP}. 
\begin{lma}\label{lma1}
	Let $\mcl$ be a non trivial nearly $T_B^{-1}$ invariant subspace of $\dcl_\alpha$ with defect $p$ for $\alpha \in [-1,0)$ and let $\fcl$ be the corresponding $p$ dimensional defect space. Let $\{f_i \}_{i=1}^r$ and $\{e_j\}_{j=1}^p$ be an orthonormal basis of $\mcl \ominus (\mcl\cap T_B\dcl_\alpha )$ and $\fcl$ respectively. Then for all $f\in \mcl $, there exist $\{q_i \}_{i=1}^r$ and $\{h_j\}_{j=1}^p$ in $\mathcal{O}(s\D)$ such that 
	\begin{equation}\label{fsb1}
	f=\sum_{i=1}^{r}f_iq_i +\gamma_1^{-1}T_B\sum_{j=1}^{p}e_jh_j \quad \text{on}~ s\D,
	\end{equation}
	for all $i\in \{1,2,\ldots ,r \}$ and $j\in \{1,2,\ldots ,p \}$, and also there exist $(a_{ki})_{k\in \N _0} \in \C^\N$ and $(b_{kj})_{k\in \N }\in \C^\N$  with 
	\begin{align}
	&q_i =\sum_{k=0}^{\infty}a_{ki}\bigg(\gamma_1^{-1}B\bigg)^k \quad \text{on}~s\D, \quad  h_j=\sum_{k=1}^{\infty}b_{kj}\bigg( \gamma_1^{-1}B\bigg)^{k-1} \quad \text{on}~s\D \label{eq}
	\end{align}
	and 
	\begin{align}
	&\sum_{i=1}^{r}\sum_{k=0}^{\infty}|a_{ki}|^2 +\sum_{j=1}^{p}\sum_{k=1}^{\infty}|b_{kj}|^2 \leq \norm f_{\dcl_\alpha}^2 \label{ineq1}.
	\end{align}
	\end{lma}
\begin{rmrk}\label{rem1}
	If the subspace $\mcl \subseteq T_B\dcl_\alpha$, then using the same notation as in Lemma \ref{lma1}, for all $f\in \mcl $ there exists $\{h_j\}_{j=1}^p$ in $\mathcal{O}(s\D)$ such that 
	\begin{equation*}
	f=\gamma_1^{-1}T_B\sum_{j=1}^{p}e_jh_j ~on~ s\D,
	\end{equation*}
	 and also there exists $(b_{kj})_{k\in \N }\in \C^\N$  with 
	\begin{align*}
	 h_j &=\sum_{k=1}^{\infty}b_{kj}\bigg( \gamma_1^{-1}B\bigg)^{k-1} \quad \text{and} \quad
     \sum_{j=1}^{p}\sum_{k=1}^{\infty}|b_{kj}|^2 \leq \norm f_{\dcl_\alpha}^2 .
	\end{align*}
\end{rmrk}
Here our main aim is to describe the nearly $T_B^{-1}$ invariant subspaces of $\dcl_\alpha$ with finite defect for $\alpha\in [-1,0)$ in terms of $T_{B^{-1}}$ invariant subspaces of $H^2_{\C^{r+p}}(s\D)$. In order to get a connection with invariant subspaces of $H^2_{\C^{r+p}}(\D)$ we introduce an unitary mapping $U_s :H^2_{\C^{r+p}}(s\D)\to H^2_{\C^{r+p}}(\D)$ by $$(U_sf)(z)=f(sz).$$
If we denote $T_s^* := U_sT_{B^{-1}}U_s^*$, then we have the following commutative diagram \ref{diag2}
\begin{align}\label{diag2}
\begin{matrix}
H^2_{\C^{r+p}}(s\D)\xrightarrow[\hspace*{3cm}]{T_B^{-1}}H^2_{\C^{r+p}}(s\D)\\
U_s\Bigg\downarrow\hspace*{3.5cm}\Bigg\downarrow U_s\\
H^2_{\C^{r+p}}(\D)\xrightarrow[\hspace*{3cm}]{T_s^*}H^2_{\C^{r+p}}(\D).
\end{matrix}
\end{align}
Since the disc $s\D$ contains all the zeros of $B$, then the symbol $B^{-1}$ lies in $L^\infty(s\mathbb{T})$ and therefore by using the fact $B^{-1}(sz)=\overline{B(s^{-1}z)}$ on $\mathbb{T}$ we conclude 
\begin{equation}\label{par}
(T_s^* f)(z)= T_{\overline{B(s^{-1}z)}}f(z).
\end{equation}
For more details about \eqref{par} (see (3.18), section 3 in  \cite{YP}).
Now we state our main theorem in this subsection concerning nearly $T_B^{-1}$ invariant subspaces with defect $p$ in $\dcl_\alpha$ spaces with $\alpha \in [-1,0)$ based on above notations which gives a generalization of Theorem 3.7 in \cite{YP}.
\begin{thm}
Let $\mcl$ be a nearly $T_B^{-1}$ invariant subspace of $\dcl_\alpha$ with finite defect $p$  for $\alpha \in [-1,0)$ and let $\fcl$ be the corresponding $p$ dimensional defect space.
 Let $E_0:=[e_1,e_2,\ldots ,e_p]$, where $\{e_j\}_{j=1}^p$ is an orthonormal basis of $\fcl$ using norm $\norm ._1$. Then 
	
	(i) in the case when $\mcl \nsubseteq T_B\dcl_\alpha$, if $F_0:=[f_1,f_2,\ldots ,f_r]$ is a matrix containing an orthonormal basis $\{f_i\}_{i=1}^r$ of $\mcl \ominus(\mcl \cap T_B\dcl_\alpha)$, then there exists a linear subspace $\ncl \subset H^2_{\C^{r+p}}(s\D)$ such that 
	$$ \mcl= \Bigg\{f\in \dcl_\alpha :f=F_0 q +\gamma_1^{-1}T_BE_0h \quad \text{on} ~s\D: (q,h)\in \ncl \Bigg\} \quad \text{on} ~s\D,$$ together with $$\bigg(1-\norm{\gamma_1^{-1}B}_{H^\infty(s\D)}^2 \bigg)^{1/2}\bigg(\norm {q}^2_{H^2_{\C^r}(s\D)}+\norm {h}^2_{H^2_{\C^p}(s\D)} \bigg)^{1/2} \leq \norm {f}_{\dcl_\alpha}.$$
	Moreover, $\ncl$ is invariant under $T_B^{-1}$ and hence $U_s(\ncl)$ is invariant under $T_s^*= U_sT_{B^{-1}}U_s^*$ in $H^2_{\C^{r+p}}(\D)$.
	
	(ii) In the case when $\mcl \subset T_B\dcl_\alpha$, then there exists a linear subspace $\ncl \subset H^2_{\C^{p}}(s\D)$ such that 
	$$ \mcl= \Bigg\{f\in \dcl_\alpha :f= \gamma_1^{-1}T_BE_0h : h\in \ncl \Bigg\} \quad \text{on} ~s\D,$$ together with $$\bigg(1-\norm{\gamma_1^{-1}B}_{H^\infty(s\D)}^2 \bigg)^{1/2}\norm {h}_{H^2_{\C^p}(s\D)}  \leq \norm {f}_{\dcl_\alpha}.$$
	Moreover, $\ncl$ is invariant under $T_B^{-1}$ and hence $U_s(\ncl)$ is invariant under $T_s^*= U_sT_{B^{-1}}U_s^*$ in $H^2_{\C^{p}}(\D)$  (Note that here $U_s :H^2_{\C^{p}}(s\D) \to H^2_{\C^{p}}(\D)$).
	   
\end{thm}
\begin{proof}
(i) For $f\in \mcl\subset \dcl_\alpha$ with $\alpha \in [-1,0)$, the equation \eqref{fsb1} in the above Lemma \ref{lma1} implies
\begin{align}\label{mainthmeq1}
f &=\sum_{i=1}^{r}f_iq_i +\gamma_1^{-1}T_B\sum_{j=1}^{p}e_jh_j 
=F_0 q +\gamma_1^{-1}T_BE_0h, \quad \text{on} ~s\D
\end{align}
where $q=[q_1,q_2,\ldots ,q_r]^t$ and $h=[h_1,h_2,\ldots ,h_p]^t$.	
Using the facts \ref{inq} and \ref{eq} we  obtain the following for all $i\in \{1,2,\ldots ,r\}$ and $j\in\{1,2,\ldots ,p\}$,
\begin{align*}
\norm {q_i}_{H^2(s\D)}&=\norm {\sum_{k=0}^{\infty}a_{ki}(\gamma_1^{-1}B)^k}_{H^2(s\D)}
\leq \sum_{k=0}^{\infty}|a_{ki}|\norm {\gamma_1^{-1}B}^k_{H^\infty(s\D)}\\
&\leq \bigg(\sum_{k=0}^{\infty}\norm{\gamma_1^{-1}B}^{2k}_{H^\infty(s\D)}\bigg)^{1/2} \bigg(\sum_{k=0}^{\infty}|a_{ki}|^2\bigg)^{1/2}
= \bigg(1-\norm{\gamma_1^{-1}B}^{2}_{H^\infty(s\D)}\bigg)^{-1/2} \bigg(\sum_{k=0}^{\infty}|a_{ki}|^2\bigg)^{1/2},
\end{align*}
and
\begin{align*}
\norm {h_j}_{H^2(s\D)}&\leq \bigg(1-\norm{\gamma_1^{-1}B}^{2}_{H^\infty(s\D)}\bigg)^{-1/2} \bigg(\sum_{k=1}^{\infty}|b_{kj}|^2\bigg)^{1/2}.
\end{align*}
Therefore the above estimates along with  the inequality in \ref{ineq1} yields 
\begin{align*}
\norm{q} ^2_{H^2_{\C^r}(s\D)}&=\sum_{i=1}^{r}\norm{q_i}^2_{H^2(s\D)}
\leq \bigg(1-\norm{\gamma_1^{-1}B}^{2}_{H^\infty(s\D)}\bigg)^{-1} \bigg(\sum_{i=1}^{r}\sum_{k=0}^{\infty}|a_{ki}|^2\bigg)\\
&\leq \bigg(1-\norm{\gamma_1^{-1}B}^{2}_{H^\infty(s\D)}\bigg)^{-1}\norm{f}_{\dcl_\alpha}^2 < +\infty,
\end{align*}
and
\begin{align*}
\norm{h} ^2_{H^2_{\C^p}(s\D)}&\leq \bigg(1-\norm{\gamma_1^{-1}B}^{2}_{H^\infty(s\D)}\bigg)^{-1}\norm{f}_{\dcl_\alpha}^2 < +\infty.
\end{align*}
Thus the above implies
\begin{align*}
q&=\sum_{k=0}^{\infty}A_k(\gamma_1^{-1}B)^k \in H^2_{\C^r}(s\D) \quad \text{where}~ A_k=[a_{k1},a_{k2},\ldots ,a_{kr}]^t ,
\end{align*}
and
\begin{align*}
 h&=\sum_{k=1}^{\infty}B_k(\gamma_1^{-1}B)^{k-1} \in H^2_{\C^p}(s\D) \quad \text{where}~ B_k=[b_{k1},b_{k2},\ldots ,b_{kp}]^t.
 \end{align*}
 Moreover, the equation \ref{ineq1} implies for all $f\in \mcl$, 
 \begin{align}
 \norm{q} ^2_{H^2_{\C^r}(s\D)}+\norm{h} ^2_{H^2_{\C^p}(s\D)}&\leq \bigg(1-\norm{\gamma_1^{-1}B}^{2}_{H^\infty(s\D)}\bigg)^{-1}\norm{f}_{\dcl_\alpha}^2. 
 \end{align}
 Now we define a linear subspace as follows: $$\ncl :=\Bigg\{(q,h)\in  H^2_{\C^{r}}(s\D)\times H^2_{\C^{p}}(s\D) :\exists f\in \mcl, ~f=F_0 q +\gamma_1^{-1}T_BE_0h \quad \text{on}~ s\D   \Bigg\} ,$$
 satisfying for any $f\in \mcl $, $\exists (q,h)\in \ncl$ such that $f=F_0 q +\gamma_1^{-1}T_BE_0h$ on $s\D$. Next we show that $\ncl$ is invariant under $T_{B^{-1}}$. By considering $T=\gamma_1^{-1}T_B$ in Lemma \ref{lm}, the equation  \eqref{eqlmma1} with $m=0$ implies
 \begin{align*}
 f&=Qf+TRf=Qf+\gamma_1^{-1}T_BRf.
 \end{align*}
Moreover, on $s\D$, the above equation together with \eqref{mainthmeq1} yields that
\begin{align*}
F_0 q +\gamma_1^{-1}T_BE_0h &=Q(F_0 q +\gamma_1^{-1}T_BE_0h)+\gamma_1^{-1}T_BR(F_0 q +\gamma_1^{-1}T_BE_0h)\\
&=F_0A_0 +\gamma_1^{-1}BR(F_0 q +\gamma_1^{-1}T_BE_0h),
\end{align*}
which further satisfies 
\begin{align*}
F_0 (q-A_0)+\gamma_1^{-1}T_BE_0h &=\gamma_1^{-1}BR(F_0 q +\gamma_1^{-1}T_BE_0h).
\end{align*}
Next by using the fact  $T_B$ is injective, we conclude from the above that
\begin{align}\label{eqmth1}
\gamma_1^{-1}R(F_0 q +\gamma_1^{-1}T_BE_0h) &= F_0 (\sum_{k=1}^{\infty}A_k\gamma_1^{-k}B^{k-1}) +\gamma_1^{-1}E_0h=F_0(T_{B^{-1}}q)+\gamma_1^{-1}E_0h.
\end{align}
Moreover, by using the fact that $R(F_0 q +\gamma_1^{-1}T_BE_0h)\in \mcl \oplus \fcl$ we obtain
\begin{align}\label{eqmth2}
R(F_0 q +\gamma_1^{-1}T_BE_0h) &= P^\mcl R(F_0 q +\gamma_1^{-1}T_BE_0h) +E_0 B_1.
\end{align} 
Thus by combining equations \eqref{eqmth1} and \eqref{eqmth2} we get 

\begin{align*}
\gamma_1 ^{-1}P_\mcl R(F_0 q +\gamma_1^{-1}T_BE_0h) &=F_0(T_{B^{-1}}q) +\gamma_1 ^{-1}E_0(\sum_{k=2}^{\infty}B_k(\gamma_1^{-1}B)^{k-1})\\
&=F_0(T_{B^{-1}}q)+\gamma_1 ^{-1}T_BE_0(T_{B^{-1}}h).
\end{align*}
Note that $\gamma_1 ^{-1}P^\mcl R(F_0 q +\gamma_1^{-1}T_BE_0h) \in \mcl$ tand hence from the definition of $\ncl$ we conclude $(T_{B^{-1}}q,T_{B^{-1}}h)\in \ncl$. Thus $\ncl$ is $T_{B^{-1}}$ invariant in $H^2_{\C^{r+p}}(s\D)$. Finally, by using the diagram \ref{diag2} we have $T_s^*(U_s(\ncl))\subset U_s(\ncl) $, that is $U_s(\ncl)$ is invariant under $T_s^*$. 

(ii) If $\mcl \subset T_B\dcl_\alpha$, then by using Remark \ref{rem1} and proceeding as in case $(i)$ we  obtain a linear subspace $\ncl \subset H^2_{\C^{p}}(s\D)$ such that 
$$ \mcl= \Bigg\{f\in \dcl_\alpha :f= \gamma_1^{-1}T_BE_0h : h\in \ncl \Bigg\} \quad \text{on}~s\D,$$ together with $$\bigg(1-\norm{\gamma_1^{-1}B}_{H^\infty(s\D)}^2 \bigg)^{1/2}\norm {h}_{H^2_{\C^p}(s\D)}  \leq \norm {f}_{\dcl_\alpha}.$$
Moreover, $\ncl$ is invariant under $T_B^{-1}$ and $U_s(\ncl)$ is invariant under $T_s= U_sT_{B^{-1}}U_s^*$ in $H^2_{\C^{p}}(\D)$. (Note that here $U_s :H^2_{\C^{p}}(s\D) \to H^2_{\C^{p}}(\D)$).  This completes the proof. 
\end{proof}

\subsection{\boldmath$\alpha \in [0,1]$}: Here we consider $\dcl_\alpha$ spaces with $\alpha \in [0,1]$ and $B$ is a finite Blaschke product of degree $m$. We now endow $\dcl_\alpha$ with the following equivalent norm denoted by $\norm ._2$ and is defined by 
\begin{equation}
\norm {f}_2^2:=\sum_{n=0}^{\infty}(n+1)^\alpha\norm{g_n}^2_{\hdcc}
\end{equation}
for any $f=\sum\limits_{n=0}^{\infty}g_nB^n$ with $g_n \in \mathcal{K}_B$ (see Theorem \ref{th1}). Therefore we have,
\begin{align*}
\norm {T_Bf}_2^2=\norm {Bf}_2^2=\sum_{n=0}^{\infty}(n+2)^\alpha \norm {g_n}^2_{\hdcc}\geq \norm{f}_2^2
\end{align*}
which implies that the operator $T_B:(\dcl_\alpha,\norm ._2)\to (\dcl_\alpha ,\norm ._2) $ is lower bounded and the lower bound  \ref{lbd} of  $T_B$ relative to the norm $\norm ._2$ is $\gamma_2 :=1$. Moreover, the pair $(\dcl_\alpha ,B)$ also satisfies the conditions (i)-(iv). Furthermore it is easy to check that $B^{-1}(D(0,1))=B^{-1}(\D)=\D$ and $\bigcap_{m\in \N}B^m\dcl_\alpha =\{0\}$ on $\D$. These facts along with Theorem \ref{thm1} (with $\hilh =\dcl_\alpha, u=B, \gamma =\gamma_2 =1$ and $I=\{1,2,\ldots ,r \}$) gives the following lemma which is a generalization of Lemma 3.3. in \cite{YP}. 
\begin{lma}\label{lma2}
		Let $\mcl$ be a non trivial nearly $T_B^{-1}$ invariant subspace of $\dcl_\alpha$ for $\alpha \in [0,1]$ such that $\mcl \nsubseteq T_B\dcl_\alpha$ and let $\{f_i \}_{i=1}^r$  and $\{e_j\}_{j=1}^p$ be an orthonormal basis of $\mcl \ominus (\mcl\cap T_B\dcl_\alpha )$  and the defect space $\fcl$ respectively. Then for any $f\in \mcl $, there exist $\{q_i \}_{i=1}^r$ and $\{h_j\}_{j=1}^p$ in $\mathcal{O}(\D)$ such that 
	\begin{equation*}
	f=\sum_{i=1}^{r}f_iq_i +T_B\sum_{j=1}^{p}e_jh_j 
	\end{equation*}
	for any $i\in \{1,2,\ldots ,r \}$; $j\in \{1,2,\ldots ,p \}$ and also there exist $(c_{ki})_{k\in \N _0} \in \C^\N$ and $(d_{kj})_{k\in \N }\in \C^\N$  with 
	\begin{align}
	&q_i =\sum_{k=0}^{\infty}c_{ki} B^k, h_j=\sum_{k=1}^{\infty}d_{kj} B^{k-1} \label{eq2}\end{align}
	and 
	\begin{align}
	&\sum_{i=1}^{r}\sum_{k=0}^{\infty}|c_{ki}|^2 +\sum_{j=1}^{p}\sum_{k=1}^{\infty}|d_{kj}|^2 \leq \norm f_{\dcl_\alpha}^2 \label{ineq2}.
	\end{align}
\end{lma}
\begin{rmrk}\label{rem2}
	If $\mcl \subseteq T_B\dcl_\alpha$, then using the same notation as in Lemma \ref{lma2} for any $f\in \mcl $ there exists $\{h_j\}_{j=1}^p$ in $\mathcal{O}(\D)$ such that 
	\begin{equation*}
	f=T_B\sum_{j=1}^{p}e_jh_j
	\end{equation*}
	and also there exists $(b_{kj})_{k\in \N }\in \C^\N$  with 
	\begin{align}
	h_j &=\sum_{k=1}^{\infty}b_{kj}B^{k-1} \quad \text{and} \quad 
	\sum_{j=1}^{p}\sum_{k=1}^{\infty}|b_{kj}|^2 \leq \norm f_{\dcl_\alpha}^2 .
	\end{align}
\end{rmrk}
Now we are in a position to describe the nearly $T_B ^{-1} $ invariant subspace with defect $p$ in $\dcl_\alpha$ for $\alpha \in [0,1]$, providing a generalization of Theorem 3.4 in \cite{YP}. Due to Lemma \ref{a1} without loss of generality we assume $B(0)=0$.
\begin{thm}\label{thm2}
	Let $\mcl$ be a nearly $T_B^{-1}$ invariant subspace of $\dcl_\alpha$ with finite defect $p$  for $\alpha \in [0,1]$ and let $\fcl$ be the $p$ dimensional defect space. Let $E_0:=[e_1,e_2,\ldots ,e_p]$ where $\{e_j\}_{j=1}^p$ is an orthonormal basis of $\fcl$ using norm $\norm ._2$. Then 
	
	(i) in the case when $\mcl \nsubseteq T_B\dcl_\alpha$, if $F_0:=[f_1,f_2,\ldots ,f_r]$ is a matrix containing an orthonormal basis $\{f_i\}_{i=1}^r$ of $\mcl \ominus(\mcl \cap T_B\dcl_\alpha)$, then there exists a linear subspace $\ncl \subset H^2_{\C^{r+p}}(\D)$ such that 
	$$ \mcl= \Bigg\{f\in \dcl_\alpha :f=F_0 q +T_BE_0h : (q,h)\in \ncl \Bigg\}$$  together with 
	$$\norm {q}^2_{H^2(\D,\C^r)}+\norm {h}^2_{H^2(\D,\C^p)}\leq \norm {f}^2_{\dcl_\alpha}.$$ Moreover, $\ncl$ is $T_{\overline{B}}$ invariant.
	
	(ii) In the case $\mcl \subset T_B\dcl_\alpha$, there exists a linear subspace $\ncl \subset H^2_{\C^{p}}(\D)$ such that 
	$$ \mcl= \Bigg\{f\in \dcl_\alpha :f= T_BE_0h : h\in \ncl \Bigg\}$$  together with $$\norm {h}^2_{H^2(\D,\C^p)}\leq \norm {f}^2_{\dcl_\alpha}, $$ and $\ncl$ is $T_{\overline{B}}$ invariant.
\end{thm}
\begin{proof}
(i) For $f\in \mcl\subset \dcl_\alpha$ with $\alpha \in [0,1]$, then by applying Lemma \ref{lma2} we get
\begin{align}\label{fnthm1}
f &=\sum_{i=1}^{r}f_iq_i +T_B\sum_{j=1}^{p}e_jh_j =F_0 q +T_BE_0h,  
\end{align}
where $q=[q_1,q_2,\ldots ,q_r]^t$ and $h=[h_1,h_2,\ldots ,h_p]^t$.
Next by using the facts \ref{eq2} and \ref{ineq2} we obtain the following norm equalities and norm estimates for any $i\in \{1,2,\ldots ,r\}$ and $j\in\{1,2,\ldots ,p\}$:	
\begin{align*}
\norm {q_i}_{\hdcc}^2=\sum_{k=0}^{\infty}|c_{ki}|^2, \quad
\norm {h_j}_{\hdcc}^2=\sum_{k=1}^{\infty}|d_{ki}|^2,
\end{align*}
and hence
\begin{align*}
\norm{q} ^2_{H^2_{\C^r}(\D)} + \norm{h} ^2_{H^2_{\C^p}(\D)} =\sum_{i=1}^{r}\norm{q_i}^2_{H^2_{\C}(\D)} + \sum_{j=1}^{p}\norm{h_j}^2_{H^2_{\C}(\D)} =\sum_{i=1}^{r}\sum_{k=0}^{\infty}|c_{ki}|^2
+ \sum_{j=1}^{p}\sum_{k=1}^{\infty}|d_{kj}|^2\leq \norm{f}_{\dcl_\alpha}^2.
\end{align*}

Thus it follows that 
\begin{align*}
q&=\sum_{k=0}^{\infty}C_kB^k \in H^2_{\C^r}(\D), \quad \text{where}~ C_k=[c_{k1},c_{k2},\ldots ,c_{kr}]^t,
\end{align*} 
and
\begin{align*}
h&=\sum_{k=1}^{\infty}D_kB^{k-1} \in H^2_{\C^p}(\D) \quad \text{where}~ D_k=[d_{k1},d_{k2},\ldots ,d_{kp}]^t.
\end{align*} 
Now we define a linear subspace as follows $$\ncl :=\Bigg\{(q,h)\in  H^2_{\C^{r}}(\D)\times H^2_{\C^{p}}(\D) :\exists f\in \mcl \quad \text{such that}\quad f=F_0 q +T_BE_0h  \Bigg\},$$
satisfying for any $f\in \mcl $, $\exists (q,h)\in \ncl$ such that
\begin{align*}
f =F_0 q +T_BE_0h \quad \text{with}\quad 
\norm{f}_{\dcl_\alpha}^2 &\geq  \norm{q} ^2_{H^2_{\C^r}(\D)} + \norm{h} ^2_{H^2_{\C^p}(\D)}
\end{align*}
 Next we show that $\ncl$ is invariant under $T_{\overline{B}}$. Consider $T=T_B$ and $\hilh =\dcl_\alpha$ for $\alpha\in[0,1]$ in Lemma \ref{lm} and therefore the corresponding operator $R$, $Q$ and $S$ in Lemma \ref{lm} becomes $R=(T_B^*T_B)^{-1}T_B^*P_{\mcl \cap T_B\dcl_\alpha}, Q=P_{\mcl \ominus (\mcl \cap \dcl_\alpha)}, S=P_\fcl$ and hence the equation  \eqref{eqlmma1} with $m=0$ implies
 for any $f\in \mcl$, 
\begin{align*}
f =Qf+TRf
=Qf+T_BRf,
\end{align*}  
which together with \eqref{fnthm1} yields
\begin{align*}
F_0 q +T_BE_0h &=Q(F_0 q +T_BE_0h)+T_BR(F_0 q +T_BE_0h)\\
&=F_0C_0 +BR(F_0 q +T_BE_0h),
\end{align*}
which further satisfies
\begin{align*}
F_0 (q-C_0)+T_BE_0h &=BR(F_0 q +T_BE_0h).
\end{align*}
Since $T_B$ is injective, then from the above we conclude
\begin{align}\label{finthm2}
R(F_0 q +T_BE_0h) &= F_0 (\sum_{k=1}^{\infty}C_kB^{k-1}) +E_0h
=F_0(T_{\overline{B}}q)+E_0h.
\end{align}
On the other hand note that $R(F_0 q +T_BE_0h)\in \mcl \oplus \fcl$ and hence
\begin{align}\label{finthm3}
R(F_0 q +T_BE_0h) &= P_\mcl R(F_0 q +T_BE_0h) +E_0 D_1
\end{align} 
Thus by combining \eqref{finthm2} and \eqref{finthm3} we get 

\begin{align*}
P_\mcl R(F_0 q +T_BE_0h) &=F_0(T_{\overline{B}}q) + E_0(\sum_{k=2}^{\infty}D_kB^{k-1})
=F_0(T_{\overline{B}}q)+T_BE_0(T_{\overline{B}}h).
\end{align*}
Since $P_\mcl R(F_0 q +T_BE_0h) \in \mcl$, then from the definition of $\ncl$ it follows that $(T_{\overline{B}}q,T_{\overline{B}}h)\in \ncl$. Thus $\ncl$ is $T_{\overline{B}}$ invariant in $H^2_{\C^{r+p}}(\D)$.

(ii) If $\mcl \subset T_B\dcl_\alpha$, then by using Remark \ref{rem2} and proceeding similarly as in case (i) we  obtain a linear subspace $\ncl \subset H^2_{\C^{p}}(\D)$ such that 
$$ \mcl= \Bigg\{f\in \dcl_\alpha :f= T_BE_0h : h\in \ncl \Bigg\} \quad  \text{together with} \quad  \norm {h}_{H^2_{\C^p}(s\D)}  \leq \norm {f}_{\dcl_\alpha},$$ and $\ncl$ is $T_{\overline{B}}$ invariant in $H^2_{\C^{p}}(\D)$. This completes the proof.
\end{proof}
Next we consider a special case of \ref{diag1}
\begin{align}\label{diag3}
\begin{matrix}
H^2_{\C^{r+p}}(\D)\xrightarrow[\hspace*{3cm}]{T}H^2_{\C^{r+p}}(\D)\\
U\Bigg\downarrow\hspace*{3.5cm}\Bigg\downarrow U\\
H^2_{\C^{m(r+p)}}(\D)\xrightarrow[\hspace*{3cm}]{S}H^2_{\C^m({r+p})}(\D)
\end{matrix}
\end{align}
Then $SU=UT_B$ holds for the unilateral shift $S:H^2_{\C^{m(r+p)}}(\D)\to H^2_{\C^{m(r+p)}}(\D)$ and $T_B:H^2_{\C^{r+p}}(\D) \to H^2_{\C^{r+p}}(\D)$ having multiplicity $m(r+p)$. Using this fact we  have the following remark concerning finite dimensional nearly $T_B^{-1}$ invariant subspaces of $\dcl_\alpha$ for $\alpha \in [0,1]$.
\begin{rmrk}
	Note that the subspace $\ncl$ is not closed in general. In the above Theorem \ref{thm2} if we consider $\mcl $ is finite dimensional, then $\ncl \subset H^2_{\C^{r+p}}(\D)$ is also finite dimensional and hence closed. Then from Beurling-Lax-Halmos Theorem and using diagram \ref{diag3} we  obtain that there exists a non negative integer $l$ with $l\leq m(r+p)$ and an inner multiplier $\Phi \in H^\infty_{\mathcal{L}(\C^l,\C^{m(r+p)})}(\D)$ such that
	\begin{align*}
	\ncl&= U^*\bigg(H^2_{\C^{m(r+p)}}(\D) \ominus \Phi H^2_{\C^l}(\D)\bigg) \quad \text{and hence} ~\\
	\mcl &=\Bigg\{f\in \dcl_\alpha :f=F_0 q +T_BE_0h : (q,h)\in U^*\bigg(H^2_{\C^{m(r+p)}}(\D) \ominus \Phi H^2_{\C^l}(\D)\bigg)\Bigg\}.
	\end{align*}
\end{rmrk}

\end{document}